\documentclass[11pt]{article}
\usepackage[left=1.25in, right=1.25in, top=1.25in, bottom=1.25in]{geometry}
\usepackage{amsmath, amsfonts,amssymb, amsthm}
\usepackage{graphicx}
\usepackage{caption}
\usepackage{subcaption}
\usepackage{color}
\usepackage{epsf}
\usepackage{tikz}
\usepackage[normalem]{ulem} 
\usepackage{hyperref}

\hypersetup{
            breaklinks=true,%
            pdfstartpage=1,%
			plainpages=false,%
			naturalnames=true,%
linkbordercolor=blue,
               urlcolor=blue,%
			citecolor=blue,%
			pdftitle={Project Description},%
			pdfauthor={Andrzej Dudek}}

\newtheorem{theorem}{Theorem}[section]
\newtheorem{corollary}[theorem]{Corollary}
\newtheorem{lemma}[theorem]{Lemma}

\newtheorem{observation}[theorem]{Observation}

\newcommand{\ep}{\epsilon}

\newcommand{\gnp}{\mathcal{G}(n,p)}
\newcommand{\hnp}{\mathcal{H}^{(k)}(n,p)}

\newcommand{\mc}{\mathrm{mc}}

\newcommand{\rbrac}[1]{\left( #1 \right)}

\newcommand{\cbrac}[1]{\left\{ #1 \right\}}

\title{Large monochromatic components and long monochromatic cycles in random hypergraphs}

\author{Patrick Bennett\thanks{Department of Mathematics, Western Michigan University, Kalamazoo, MI. Supported in part by Simons Foundation Grant \#426894.}
\and Louis DeBiasio\thanks{Department of Mathematics, Miami University, Oxford, OH. Supported in part by Simons Foundation Grant \#283194.}
\and Andrzej Dudek\thanks{Department of Mathematics, Western Michigan University, Kalamazoo, MI. Supported in part by Simons Foundation Grant \#522400.}
\and Sean English\thanks{Department of Mathematics, Western Michigan University, Kalamazoo, MI.}}

\date{\today}

\begin{document}
\maketitle

\begin{abstract}
We extend results of Gy\'arf\'as and F\"uredi on the largest monochromatic component in $r$-colored complete $k$-uniform hypergraphs to the setting of random hypergraphs.  We also study long monochromatic loose cycles in $r$-colored random hypergraphs.  In particular, we obtain a random analog of a result of Gy\'arf\'as, S\'ark\"ozy, and Szemer\'edi on the longest monochromatic loose cycle in $2$-colored complete $k$-uniform hypergraphs.
\end{abstract}

\section{Introduction}

It is known, due to Gy{\'a}rf{\'a}s~\cite{G77}, that for any $r$-coloring of the edges of $K_n$, there is a monochromatic component of order at least $n/(r-1)$ and this is tight when $r-1$ is a prime power and $n$ is divisible by $(r-1)^2$.  Later, F\"uredi~\cite{F81} introduced a more general method which implies the result of Gy\'arf\'as.  More recently, Mubayi \cite{Mub2002} and independently, Liu, Morris, and Prince \cite{LMP09}, gave a simple proof of a stronger result which says that in any $r$-coloring of the edges of $K_n$, there is either a monochromatic component on $n$ vertices or a monochromatic double star of order at least $n/(r-1)$.  Recently, Bal and DeBiasio~\cite{BD2017} and independently, Dudek and Pra\l{}at~\cite{DP2017}, showed that the Erd\H{o}s-R\'enyi random graph behaves very similarly with respect to the size of the largest monochromatic component. 

Recall that the \emph{random graph} $\gnp$ is the random graph $G$ with vertex set $[n]$ in which every pair $\{i,j\} \in \binom{[n]}{2}$ appears independently as an edge in $G$ with probability~$p$. An event in a probability space holds \emph{asymptotically almost surely} (or \emph{a.a.s.}) if the probability that it holds tends to $1$ as $n$ goes to infinity.  More precisely, it was shown in~\cite{BD2017} and \cite{DP2017} that that a.a.s.\ for any $r$-coloring of the edges of $\gnp$, there is a monochromatic component of order $(\frac{1}{r-1}-o(1))n$, provided that $pn \to \infty$ (that means the average degree tends to infinity). As before, this result is clearly best possible.

In this paper we study a generalization of these results to $k$-uniform hypergraphs (each edge has order~$k$).  As in the $k=2$ case, our results hold even for very sparse random hypergraphs; that is, we only assume that the average degree, $pn^{k-1}$, tends to infinity together with~$n$. 

\subsection{Large components}

We say that a hypergraph $H=(V,E)$ is \emph{connected} if the \emph{shadow graph} of $H$ (that is, the graph with vertex set $V(H)$ and edge set $\{\{x,y\}: \{x,y\}\subseteq e \text{ for some } e\in E(H)\}$) is connected.  A \emph{component} of a hypergraph is a maximal connected subgraph.  Let $r$ be a positive integer and $H$ be a hypergraph. Let $\chi_r:E(H) \to [r]$ be a coloring of the edges of~$H$. Denote by $\mc(H,\chi_r)$ the order of the largest monochromatic component under $\chi$ and let
\[
\mc_r(H) = \min_{\chi_r} \mc(H,\chi_r).
\] 

For hypergraphs much less is known; however, Gy\'arf\'as~\cite{G77} (see also F\"uredi and Gy\'arf\'as \cite{FG1991}) proved the following result. Let $K^k_n$ denote the complete $k$-uniform hypergraph of order~$n$.

\begin{theorem}[Gy\'arf\'as \cite{G77}]\label{gthm1}
For all $n\geq k\geq 3$, 
\begin{enumerate}
\item $\mc_{k}(K^k_n)=n$, and 
\item  $\mc_{k+1}(K^k_n)\geq \frac{k}{k+1}n$. Furthermore, this is optimal when $n$ is divisible by $k+1$. 
\end{enumerate}
\end{theorem}
\noindent
The optimality statement is easy to see. Indeed, split the vertex set into $k+1$ parts $V_1, \dots, V_{k+1}$ each of size $n/(k+1)$ and color the edges so that color $i$ is not used on any edge which intersects $V_i$. Clearly each monochromatic component has size $kn/(k+1)$.

F\"uredi and  Gy\'arf\'as \cite{FG1991} and Gy\'arf\'as and Haxell \cite{GH2009} proved a number of other results regarding the value of $\mc_{r}(K^k_n)$; however, in general, the value of $\mc_{r}(K^k_n)$ is not known (see Section \ref{sec:conc} for more details).

Our main theorem shows that in order to prove a random analog of any such result about $\mc_{r}(K^k_n)$ it suffices to prove a nearly complete (or large minimum degree) analog of such a result. 
Let the \emph{random hypergraph} $\hnp$ be the $k$-uniform hypergraph $H$ with vertex set $[n]$ in which every $k$-element set from $\binom{[n]}{k}$ appears independently as an edge in~$H$ with probability~$p$. 

\begin{theorem}\label{thm:main1}
Suppose a function $\varphi = \varphi(r, k)$ satisfies the following condition: for all $k\geq 2$, $r\geq 1$, and $\alpha^*> 0$ there exists $\ep^*>0$ and $t_0>0$ such that if $G$ is a $k$-uniform hypergraph on $t\geq t_0$ vertices with $e(G)\geq (1-\ep^*)\binom{t}{k}$, then $\mc_r(G)\geq (\varphi-\alpha^*) t$.

Then for any $k\geq 2$, $r\geq 1$, $\alpha>0$, and $p=p(n)$ such that $pn^{k-1}\to\infty$ we have that a.a.s. $\mc_{r}\left(\hnp\right)\geq (\varphi-\alpha) n$.
\end{theorem}

As an application, we prove a version of Theorem \ref{gthm1} for nearly complete hypergraphs (Sections~\ref{sec:4colored} and \ref{sec:5colored}) and then obtain a version for random hypergraphs as an immediate corollary (Section~\ref{sub:components}).

\begin{theorem}\label{thm:main2}
For all $\alpha>0$ and $k\geq 3$ there exists  $\ep>0$ and $n_0$ such that if $G$ is a $k$-uniform hypergraph on $n\geq n_0$ vertices with $e(G)>(1-\ep)\binom{n}{k}$, then
\begin{enumerate}
\item $\mc_k(G)\geq (1-\alpha)n$, and
\item $\mc_{k+1}(G)\geq \left(\frac{k}{k+1}-\alpha\right)n$.
\end{enumerate}
\end{theorem}
\noindent
(We give an explicit bound on $\alpha$ in terms of $\ep$ in the proof of Theorem \ref{kunikcolors} and Corollary \ref{cor:degtoedges}.)

\begin{corollary}\label{thm:main}
Let $k\geq 3$, let $\alpha>0$, and let $p=p(n)$ be such that $pn^{k-1}\to\infty$. Then a.a.s. 
\begin{enumerate}
\item $\mc_k\left(\hnp\right)\geq (1-\alpha)n$, and 
\item $\mc_{k+1}\left(\hnp\right)\geq \left(\frac{k}{k+1}-\alpha\right)n$.
\end{enumerate}
\end{corollary}

\subsection{Loose-cycles}

We say that a $k$-uniform hypergraph $(V,E)$ is a \emph{loose cycle} if there exists a cyclic ordering of the vertices $V$ such that every edge consists of $k$ consecutive vertices and every pair of consecutive edges intersects in a single vertex (see Figure~\ref{subfig:loose}). Consequently, $|E| = |V| / (k-1)$. 

\begin{figure}
\centering
    \begin{subfigure}[b]{0.45\textwidth}
        \centering
        \includegraphics[scale=0.6]{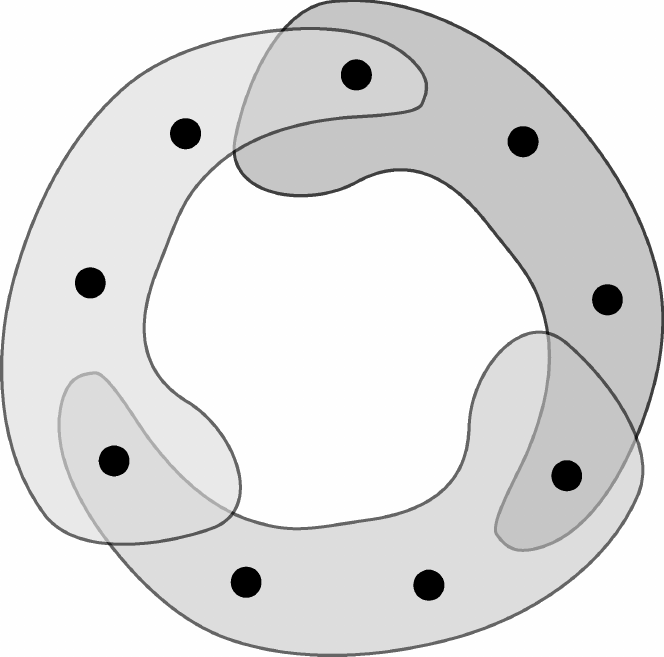}
        \caption[]{}
        \label{subfig:loose}
    \end{subfigure}%
    ~ 
    \begin{subfigure}[b]{0.45\textwidth}
        \centering
        \includegraphics[scale=0.6]{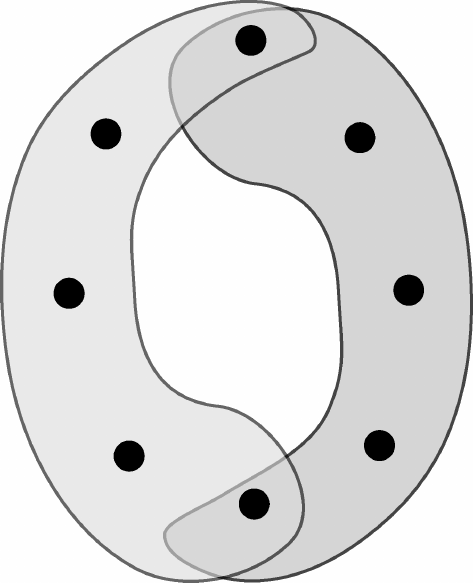}
        \caption[]{}
        \label{subfig:diamond}
    \end{subfigure}
    \caption{A 4-uniform loose cycle~(\subref{subfig:loose}) and a 5-uniform diamond~(\subref{subfig:diamond}).}
    \label{fig:loose}
\end{figure}

Let $G$ be a $k$-uniform hypergraph.  A \emph{connected loose cycle packing on $t$ vertices} is a connected sub-hypergraph $H\subseteq G$ and a vertex disjoint collection of loose cycles $C_1, \dots, C_\ell\subseteq H$ such that $\sum_{i=1}^\ell|V(C_i)|=t$.  A \emph{connected diamond matching on $t$ vertices} is a connected loose cycle packing on $t$ vertices such that every cycle consists of 2 edges and so has exactly $2k-2$ vertices (see Figure~\ref{subfig:diamond}). When $k=2$ we consider an edge to be a cycle with $2$ vertices.

Gy{\'a}rf{\'a}s, S{\'a}rk{\"o}zy and Szemer{\'e}di \cite{GSS2008diamond} proved that for all $k\geq 3$ and $\eta>0$, there exists $n_0$ such that if $n\geq n_0$, then every $2$-edge-coloring of $K_n^k$ contains a monochromatic cycle of length $(1-\eta)\frac{2k-2}{2k-1}n$. (The $k=3$ case was previously proved by Haxell, \L{}uczak, Peng, R\"odl, Ruci\'nski, Simonovits, and Skokan~\cite{HLPRRSS2006}.) Their proof follows from a more basic result which says that for all $k\geq 3$ and $\eta>0$, there exists $\ep>0$ and $n_0$ such that if $G$ is a $k$-uniform hypergraph on $n\geq n_0$ vertices with $e(G)\geq (1-\ep)\binom{n}{k}$, then every $2$-edge-coloring of $G$ contains a monochromatic connected diamond matching on $(1-\eta)\frac{2k-2}{2k-1}n$ vertices.

We provide a result which reduces the problem of finding long monochromatic loose cycles in random hypergraphs to the problem of finding large monochromatic connected loose cycle packing in nearly complete hypergraphs (Section~\ref{sub:loose}).  Then applying the result of \cite{GSS2008diamond}, we get an asymptotically tight result for $2$-colored random $k$-uniform hypergraphs.

\begin{theorem}\label{thm:cycles1}
Suppose a function $\psi = \psi(r, k)$ satisfies the following condition: for all $k\geq 2$, $r\geq 1$, and $\alpha^*> 0$ there exists $\ep^*>0$ and $t_0>0$ such that if $G$ is a $k$-uniform hypergraph on $t\geq t_0$ vertices with $e(G)\geq (1-\ep^*)\binom{t}{k}$, then every $r$-coloring of the edges of $G$ contains a monochromatic connected loose cycle packing on at least $(\psi-\alpha^*)t$ vertices.

Then for any $k\geq 2$, $r\geq 1$, $\alpha>0$, and $p=p(n)$ such that $pn^{k-1}\to\infty$ we have that a.a.s. every $r$-edge-coloring of $\hnp$ contains a monochromatic loose cycle on at least $(\psi-\alpha)n$ vertices.
\end{theorem}

Using a result of Gy\'arf\'as, S\'ark\"ozy, and Szemer\'edi \cite{GSS2008diamond} about large monochromatic connected diamond matchings in nearly complete hypergraphs, we get the following corollary.  

\begin{corollary}\label{cor:cycles}
Let $k\geq 2$ and $\alpha>0$. Choose $p=p(n)$ such that $pn^{k-1}\to \infty$. Then a.a.s. there exists a monochromatic loose cycle on at least $\left(\frac{2k-2}{2k-1}-\alpha\right)n$ vertices in any $2$-coloring of the edges of $\mathcal{H}^{(k)}(n,p)$.
\end{corollary}

\section{Notation and definitions}

Let $G$ be a $k$-uniform hypergraph. Let $v\in V(G)$ and $U\subseteq V(G)$. Define
\[
d(v, U) = \left|   \left\{S\in \binom{U}{k-1}: S\cup \{v\}\in E(G) \right\}  \right|.
\]
Furthermore, let $\delta(G) = d(v,V(G))$, which is just the \emph{minimum degree of~$G$}. Suppose now that $v \notin U$ and define the \emph{restricted link graph of $v$}, denoted by $L(v,U)$, with vertex set $U$ and edge set $\left\{S\in \binom{U}{k-1}: S\cup \{v\}\in E(G) \right\}$. We call $L(v) = L(v,V(G)\setminus \{v\})$ the \emph{link graph of $v$}, which clearly  is a $(k-1)$-uniform hypergraph on $n-1$ vertices. 

Also recall that the \emph{1-core} of $G$ is the largest induced subgraph of $G$ with
no isolated vertices.

Now suppose that the edges of $G$ are $r$-colored.  We follow the convention that any edge~$e$ of the link graph of a vertex $v$ inherits the color of the edge $e\cup \{v\}$ in $G$.  

For expressions such as $n/t$ (for example the size of a cluster in the regularity lemma) that are supposed to be an integer, we always assume that $n/t \in \mathbb{Z}$ by rounding appropriately without affecting the argument.

\section{Random hypergraphs}

\subsection{Sparse weak hypergraph regularity}

First, we will provide a few definitions.

Let $H=(V,E)$ be a $k$-uniform hypergraph. Given pairwise-disjoint sets $U_1,\dots,U_k\subseteq V$, and $p>0$, we define $e_H(U_1,\dots,U_k)$ to be the total number of edges $e=v_1\dots v_k$ in $H$ such that $v_i\in U_i$ for all $1\leq i\leq k$. Also, we define

\[
d_{H,p}(U_1,\dots,U_k)=\frac{e_H(U_1,\dots,U_k)}{p|U_1|\cdot\ldots\cdot|U_k|}.
\]
When the host graph $H$ is clear from context, we may refer to $e_H(U_1,\dots,U_k)$ as $e(U_1,\dots,U_k)$ $d_{H,p}(U_1,\dots,U_k)$ as just $d_p(U_1,\dots,U_k)$.

For $\epsilon>0$, we say the $k$-tuple $(U_1,\dots,U_k)$ of pairwise-disjoint subsets of $V$ is \emph{$(\epsilon,p)$-regular} if
\[
|d_p(W_1,\dots,W_k)-d_p(U_1,\dots,U_k)|\leq\epsilon
\]
for all $k$-tuples of subsets $W_1\subseteq U_1,\dots,W_k\subseteq U_k$ satisfying $|W_1|\cdot\ldots\cdot|W_k|\geq \epsilon |U_1|\cdot\ldots\cdot|U_k|$.

We say that $H$ is a $(\eta,p,D)$-upper uniform hypergraph if for any pairwise disjoint sets $U_1,\dots,U_k$ with $|U_1|\ge\ldots\ge|U_k|\geq\eta|V(H)|$, $d_p(U_1,\dots,U_k)\leq D$.

The following theorem is a sparse version of weak hypergraph regularity which will be the workhorse used to prove the main result for the random hypergraph.
The sparse version of the regularity lemma~\cite{S78} for graphs was discovered independently by Kohayakawa~\cite{K97}, and R\"odl (see, for example,~\cite{C14}), and subsequently improved by Scott~\cite{Scott11}.
The following is a straightforward generalization of their result for hypergraphs, which we state here without proof.

\begin{theorem}\label{regularity}
	For any given integers $k\geq 2$, $r\geq 1$, and $t_0\geq 1$, and real numbers  $\epsilon>0$,  $D\geq 1$, there are constants $\eta=\eta(k,r,\epsilon,t_0,D)>0$, $T_0=T_0(k,r,\epsilon,t_0,D)\geq t_0$, and $N_0=N_0(k,r,\epsilon,t_0,D)$ such that any collection $H_1,\dots,H_r$ of $k$-uniform hypergraphs on the same vertex set $V$ with $|V|\geq N_0$ that are all $(\eta,p,D)$-upper-uniform with respect to density $0<p\leq 1$ admits an equipartition (i.e. part sizes differ by at most 1) of $V$ into $t$ parts with $t_0\leq t\leq T_0$ such that all but at most $\epsilon\binom{t}{k}$ of the $k$-tuples induce an $(\epsilon,p)$-regular $k$-tuple in each $H_i$.
\end{theorem}

\subsection{Large components}\label{sub:components}

We will use the following lemma to turn an $(\ep, p)$-regular $k$-tuple of some color into a large monochromatic subgraph.

\begin{lemma}\label{lem:edgetocomp}
Let $0<\ep< 1/3$, let $k\geq 2$, and let $G=(V,E)$ be a $k$-partite hypergraph with vertex partition $V=V_1\cup V_2\cup \ldots \cup V_k$.  If for all $\{U_1, \dots, U_k\}$, with $|U_i|\geq \ep |V_i|$ and $U_i\subseteq V_i$ for all $i\in [k]$, there exists an edge $\{u_1, \dots, u_k\}$ with $u_i\in U_i$ for all $i\in [k]$, then $G$ contains a connected component $H$ such that $|V(H)\cap V_i|\geq (1-\ep)|V_i|$ for all $i\in [k]$.  
\end{lemma}

\begin{proof}
Start by choosing for each $i\in [k]$ a set $X_i\subseteq V_i$ with $|X_i|=\lceil 3\ep|V_i|\rceil$ and let $G'$ be the hypergraph induced by $\{X_1,\dots, X_k\}$.  Suppose that no component of $G'$ has intersection of size at least $\ep |V_i|$ with any $X_i$.  
Choose $t \ge 2$ as small as possible so that there exist components $H_1, \dots, H_t$ of $G'$ such that for some $\ell\in [k]$, $|X_\ell\cap \bigcup_{i\in [t]}V(H_i)|\geq \ep |V_\ell|$.  By the minimality of $t$, we have that for all $j\in [k]$, $|X_j\cap \bigcup_{i\in [t]}V(H_i)|< 2\ep |V_j|$.  So let $X_\ell'=X_\ell\cap \bigcup_{i\in [t]}V(H_i)$ and for all $j\in [k]\setminus \{\ell\}$, let $X_j'=X_j\setminus \bigcup_{i\in [t]}V(H_i)$.  For all $i\in [k]$, we have $|X_i'|\geq \ep |V_i|$; however, by the construction of the sets $X_1',\dots, X_k'$, there is no edge in $G'[X_1'\cup \dots \cup X_k']$ which violates the hypothesis.

So we may assume that some component $H'$ of $G'$ intersects, say $V_1$, in at least $\ep |V_1|$ vertices. We will show that for each $i$, $|V(H')\cap X_i|\geq \ep |V_i|$. Indeed, let us assume for a moment that there exists some $j$ such that $|V(H')\cap X_j| < \ep |V_j|$ and define for each $i\in [k]$ the set $Y_i$ according to the following rule. If $|V(H')\cap X_i|\geq \ep |V_i|$, then set $Y_i=V(H')\cap X_i$; otherwise set $Y_i= X_i \setminus V(H')$.  Clearly,  $|Y_i| \ge \ep |V_i|$. By the hypothesis, there is an edge in the hypergraph $G'[Y_1, \dots, Y_k]$ induced by $\{Y_1, \dots, Y_k\}$. But this cannot happen since $Y_1 \subseteq V(H')$  and  $Y_j$ for $j\neq 1$ is disjoint from $V(H')$. 
Therefore, for each $i$, $|V(H')\cap X_i|\geq \ep |V_i|$.

Now let $H$ be the largest component of $G$ that contains $H'$.
Suppose to the contrary and without loss of generality that $|V(H)\cap V_k|<(1-\ep)|V_k|$.  But as there must be an edge in $G[V(H')\cap V_1, \dots, V(H')\cap V_{k-1}, V_k\setminus V(H)]$, this is a contradiction.  So $H$ is the desired component.  
\end{proof}

We will also make use of the following version of the Chernoff bound (see, e.g., inequality~(2.9) in~\cite{JLR}). For a binomial random variable $X$ and $\gamma\le 3/2$, we have
	\begin{equation}\label{eq:chernoff}
	\Pr(|X - \mathbf{E}(X)| \ge \gamma\mathbf{E}(X)) \le 2\exp\left( -\gamma^2 \mathbf{E}(X)/3\right).
	\end{equation}

\begin{proof}[Proof of Theorem~\ref{thm:main1}]
Let $r$, $k$, $\alpha$, and $\varphi = \varphi(r,k)$ be given. Set $\alpha^* = \alpha/2$.  Let $\epsilon^*$ and $t_0$ be the constants guaranteed by the values of $k$, $r$, and $\alpha^*$. Let $\ep < \min\{ 1/(2r), \alpha/(4 \varphi), \ep^*\}$. 
Thus, if $\Gamma$ is a $k$-uniform hypergraph on $t\geq t_0$ vertices with $e(\Gamma)\geq (1-\ep)\binom{t}{k}> (1-\ep^*)\binom{t}{k}$, then $\mc_r(\Gamma)\geq (\varphi-\alpha/2)t$. Let $H=(V,E)=\hnp$.
	
	First observe that for any fixed positive $\eta$ any sub-hypergraph $H'$ of $H$ is $(\eta,p,2)$-upper uniform. Indeed, let $U_1,\dots, U_k \subseteq V$ with $|U_1|\ge\dots\ge|U_k|\geq\eta n$ be given. Then the expected number of edges in $H$ having exactly one vertex in each $U_i$ is
$|U_1|\cdot\ldots\cdot |U_k| p \ge \eta^k n^k p$.
Thus \eqref{eq:chernoff}, applied with $\gamma=1$, implies that 
	\begin{align*}
	\Pr(e(U_1,\dots,U_k) \ge 2 |U_1|\cdot\ldots\cdot |U_k| p) &\le 2\exp\left( -|U_1|\cdot\ldots\cdot |U_k| p/3   \right) \\
	&\le 2\exp\left( -\eta^k n^kp / 3 \right) =  o(2^{-kn}),
	\end{align*}
since $pn^{k-1} \rightarrow \infty$.	Consequently, a.a.s. the number of edges in $H'$ is at most $2 |U_1|\cdot\ldots\cdot |U_k| p$. Finally, since the number of choices for $U_i$'s is trivially bounded from above by $(2^n)^k = 2^{kn}$, the union bound yields that a.a.s. $H'$ is $(\eta,p,2)$-upper uniform.

	Apply Theorem \ref{regularity} with $k$, $r$, $t_0$, $\ep$, as above, and $D=2$. Let $\eta$, $T_0$ and $N_0$ be the constants that arise and assume that $n\geq N_0$. Let $c$ be an $r$-coloring of the edges of $H$ and let $H_i$ be the subgraph of $H$ induced by the $i$th color, that means, $V(H_i)=V$ and $E(H_i)=\{e\in E\mid c(e)=i\}$. Then let $t$ be the constant guaranteed by Theorem \ref{regularity} for graphs $H_i$ and let $V_1\cup \ldots \cup V_t$ be the $(\epsilon,p)$-regular partition of $V$.
	
	Let $R$ be the $k$-uniform \emph{cluster graph} with vertex set $[t]$ where $\{i_1,\dots,i_k\}$ is an edge if and only if $V_{i_1},\dots,V_{i_k}$ form an $(\epsilon,p)$-regular $k$-tuple. Color $\{i_1,\dots,i_k\}$ in $R$ by a majority color in the $k$-partite graph $H[V_{i_1},\dots,V_{i_k}]$. Let $H'$ be the sub-hypergraph colored by this color in $H[V_{i_1},\dots,V_{i_k}]$. Observe that $d_{H',p}(V_{i_1},\dots,V_{i_k}) \ge 1/(2r)$. Indeed, the Chernoff bound~\eqref{eq:chernoff}, applied with $\gamma=1/2$, implies that a.a.s. $e(V_{i_1},\dots, V_{i_k}) \ge |V_{i_1}|\cdot\ldots\cdot |V_{i_k}| p/2$. Thus, $e_{H'}(V_{i_1},\dots, V_{i_k}) \ge |V_{i_1}|\cdot\ldots\cdot |V_{i_k}| p/(2r)$, as required.  Furthermore, since $\epsilon < 1/(2r)$, we also get that $d_{H',p}(V_{i_1},\dots,V_{i_k}) > \epsilon$.
	
Clearly, we also have $|E(R)|>(1-\epsilon)\binom{t}{k}$, so by the assumption, there is a monochromatic, say red, component of size at least $(\varphi-\alpha/2)t$
in the cluster graph. Let us assume that $\{i_{i_1},\dots,i_{i_k}\}$ is a red edge in $R$. Thus, $(V_{i_1},\dots,V_{i_k})$ is an $(\epsilon,p)$-regular $k$-tuple in~$H'$ and so
for all $k$-tuples $U_{i_1}\subseteq V_{i_1},\dots,U_{i_k}\subseteq V_{i_k}$ satisfying $|U_{i_1}|\cdot\ldots\cdot|U_{i_k}| \ge \epsilon |V_{i_1}|\cdot\ldots\cdot|V_{i_k}|$ we have
\[
|d_{H',p}(U_{i_1},\dots,U_{i_k})-d_{H',p}(V_{i_1},\dots,V_{i_k})|\leq\epsilon,
\]
which implies that 
\[
d_{H',p}(U_{i_1},\dots,U_{i_k}) > d_{H',p}(V_{i_1},\dots,V_{i_k}) - \epsilon\ge 1/(2r) - \epsilon > 0.
\]
Consequently, there exists an edge $u_{i_1}\dots u_{i_k}$ in $H'$ such that $u_{i_j}\in U_{i_j}$ for each $1\le j\le k$. Once again, this is true for any $U_{i_j}$'s satisfying $|U_{i_1}|\cdot\ldots\cdot|U_{i_k}| \ge \epsilon |V_{i_1}|\cdot\ldots\cdot|V_{i_k}|$. But the latter clearly implies that $|U_{i_j}| \ge \epsilon|V_{i_j}|$. Hence,  by Lemma \ref{lem:edgetocomp}
there exists a component in $H'$ that contains at least $(1-\ep)|V_{i_j}|$ vertices from each $V_{i_j}$. In other words, if $i_i$ is contained in a red edge in the cluster graph $R$, at least $(1-\ep)|V_i|$ vertices inside $V_i$ are in the same red component. Furthermore, since $(1-\ep)|V_i|>\frac{1}{2}|V_i|$, any two red edges that intersect in $R$ will correspond to red connected subgraphs that intersect in $H$, and thus are in the same red component in $H$.
    So, the $(\varphi-\alpha/2) t$ vertices in the largest monochromatic component in $R$ corresponds to a monochromatic component in $H$ of order at least
\begin{align*}
   (\varphi-&\alpha/2)t(1-\ep)|V_1| \geq(\varphi-\alpha/2)t(1-\ep)(1-\ep)\frac{n}{t}\\
   & = (\varphi-\alpha/2)(1-2\ep +\ep^2) n
   \ge (\varphi - \alpha/2 - 2\ep \varphi)n
    \geq(\varphi-\alpha)n,
\end{align*}
where the last inequality uses $\ep < \alpha/(4 \varphi)$.
\end{proof}

\subsection{Loose cycles}\label{sub:loose}

We first generalize a result explicitly stated by Letzter \cite[Corollary 2.1]{L15}, but independently proved implicitly by Dudek and Pra\l{}at \cite{DP15} and Pokrovskiy \cite{P14}.

\begin{lemma}\label{path in diamond}
	Let $H$ be a $k$-uniform $k$-partite graph with partite sets $X_1,\dots,X_k$ such that $|X_2|=\ldots=|X_{k-1}|=m$ and $|X_1|=|X_k|=m/2$ for some $m$. Then for all $0\leq \zeta\leq 1$ if there are no sets $U_1\subseteq X_1,\dots, U_k\subseteq X_k$ with $|U_1|=\ldots=|U_k|\geq\zeta m$ such that $H[U_1,\dots U_k]$ is empty, then there is a loose path on at least $(1-4\zeta)m-2$ edges.
\end{lemma}

\begin{proof}
The proof is based on the depth first search algorithm. A similar idea for graphs was first noticed by Ben-Eliezer, Krivelevich and Sudakov~\cite{BKS2012b, BKS2012}.

Let $\zeta$ be given. We will proceed by a depth first search algorithm with the restriction that the vertices of degree 2 in the current path are in $X_1\cup X_k$.  We will let $X_i^*$ be the current set of unexplored vertices in $X_i$. We will let $X_i'$ be the current set of vertices which were added to the path at some point, but later rejected. So, initially we have $X_i^* = X_i$ and $X_i'=\emptyset$. Start the algorithm by removing any vertex from $X_1^*$ and adding it to the path $P$.

Suppose $P$ is the current path. First consider the case where $P$ has been reduced to a single vertex, say without loss of generality $P=x_1\in X_1$. If there are no edges $\{x_1,\dots, x_k\}$ where $x_i\in X_i^*$ is unexplored for all $2\leq i\leq k$, then we move $x_1$ out of the path and back into $X_1^*$ and start the algorithm again by choosing an edge in $G[X_1^*,\dots,X_k^*]$ if we can and stopping the algorithm otherwise.  If there is an edge $\{x_1,\dots,x_k\}$ where $x_i\in X_i^*$ is unexplored for all $2\leq i\leq k$, then we add $\{x_1,\dots,x_k\}$ to $P$ and remove $x_i$ from $X_i^*$ for all $2\leq i\leq k$ and continue the algorithm from $x_k$.

Now assume $P$ has at least one edge.  Let the last edge of $P$ be $\{x_1, \dots, x_k\}$ and assume without loss of generality that $x_1$ has degree $1$ in $P$ (i.e. is the current endpoint of the path).  If there are no edges $\{x_1, y_2, \dots, y_k\}$ where $y_i\in X_i^*$ is unexplored for all $2\leq i\leq k$, then we move $x_i$ from the path to $X_i'$ for all $1\leq i\leq k-1$ and continue the algorithm from $x_k$.  If there is an edge $\{x_1, y_2, \dots, y_k\}$ where $y_i\in X_i^*$ is unexplored for all $2\leq i\leq k$, then we add $\{x_1, y_2, \dots, y_k\}$ to $P$ and remove $y_i$ from $X_i^*$ for all $2\leq i\leq k$ and continue the algorithm from $y_k$. 

Note that during every stage of the algorithm, there is no edge $\{x_1, \dots, x_k\}$ where $x_1\in X_1'$ and $x_i\in X_i^*$ for all $2\leq i\leq k$ and no edge $\{y_1, \dots, y_k\}$ where $y_k\in X_k'$ and $y_i\in X_i^*$ for all $2\leq i\leq k$.  We also know that at every stage of the algorithm, $|X_2'|=\ldots=|X_{k-1}'|=|X_1'|+|X_k'|$ since every time a vertex from either $X_1$ or $X_k$ gets rejected, so do vertices from each $X_i$, and at no step does a vertex from both $X_1$ and $X_k$ get rejected. Furthermore, 
\[
|X_2^*|=\ldots=|X_{k-1}^*|=|X_1^*|+|X_k^*|+1
\]
at every step where $P\neq \emptyset$ since each time a vertex  is removed from any $X_i^*$, for $2\leq i\leq k-1$, exactly one vertex from $X_1^*\cup X_k^*$ is removed as well, or one vertex is added to $X_1^*\cup X_k^*$ (when $P$ is a single vertex that cannot be extended) then two are removed, with the  exception of when the initial edge is selected. Note that at each step $X_i=X_i^*\cup X_i'\cup (P\cap X_i)$ where the unions are disjoint.

Notice that this algorithm cannot stop while each $|X_i^*|\geq \zeta m$, since otherwise the $X_i^*$ sets would violate the hypothesis. 

After each step of the algorithm, either the difference between $|X_1^*|$ and $|X_1'|$ decreases by 1, or the difference between $|X_k^*|$ and $|X_k'|$ decreases by 1. So there is a stage in the algorithm where either $|X_1'|=|X_1^*|$ or $|X_k'|=|X_k^*|$ but not both.  Suppose without loss of generality, we are at a stage where $|X_1'|=|X_1^*|$ and $|X_k'|<|X_k^*|$. Since $P$ always contains almost the same number of vertices in $X_1$ as it does vertices in $X_k$ (off by at most one), we have that the sums $|X_1^*|+|X_1'|$ and $|X_k^*|+|X_k'|$ differ by at most 1. This yields that 
\[
2|X_1^*| = |X_1^*| + |X_1'| \le |X_k^*| + |X_k'| +1 < 2|X_k^*|+1
\]
and so $|X_1^*| \le |X_k^*|$. Thus, $|X_2^*|=\ldots=|X_{k-1}^*|\ = |X_1^*| + |X_k^*|+1    \geq 2|X_1^*|$ for all $1\leq i\leq k$. Thus, if $|X_1'|\geq \zeta m$, and consequently, each $|X_i^*| \ge |X_1^*|  = |X_1'| \ge \zeta m$, we are done since $H[X_1',X_2^*,\dots,X_k^*]$ has no edges, a contradiction. Otherwise, if $|X_1'| < \zeta m$, then  
\[
|P\cap X_1| = |X_1| - |X_1^*| - |X_1'| = m/2 - 2|X_1'|  \ge m/2 -2\zeta m.
\]
But by our construction, each vertex in $P\cap X_1$ corresponds to two edges in the path, except at most two vertices, so we have that there are at least $m-4\zeta m -2$ edges in $P$.
\end{proof}

For a graph $G$, and hypergraph $F$ with $V(G)\subseteq V(F)$, we say $F$ is a \emph{Berge-$G$} if there is a bijection $f:E(G)\to E(F)$ such that $e\subseteq f(e)$ for all $e\in E(G)$.  A Berge path $(E_1,\dots,E_{\ell-1})$ is a \emph{Berge-$P_\ell$} where $P_\ell=(e_1,\dots,e_{\ell-1})$, and $E_i=f(e_i)$. We say $F$ contains a Berge-$G$ if $F$ contains a sub-hypergraph that is a Berge-$G$.

Note that a connected hypergraph has the property that between any two vertices, there is a Berge path.

Let $P=(E_1,\dots, E_\ell)$ be a Berge path in the cluster graph $R$ of $H$.  Assume that $P$ contains no shorter Berge path that connects the two endpoints of $P$, which we will call $V_1$ and $V_s$. Since there is no shorter Berge path, any two nonconsecutive edges $E_i$ and $E_j$ must be disjoint, and so there is a labeling $V_1, \ldots, V_s$ of the vertices of $V(P)$ such that $E_1 = \{V_1, \ldots, V_k\}$, $E_\ell = \{V_{s-k+1}, \ldots, V_s\}$, and for each $i$ all the vertices of $E_i \setminus E_{i+1}$ come before the vertices of $E_i \cap E_{i+1}$, which come before the vertices of $E_{i+1} \setminus E_i$ (see Figure~\ref{fig:berge}). 

\begin{figure}
\centering
\includegraphics[scale = 0.7]{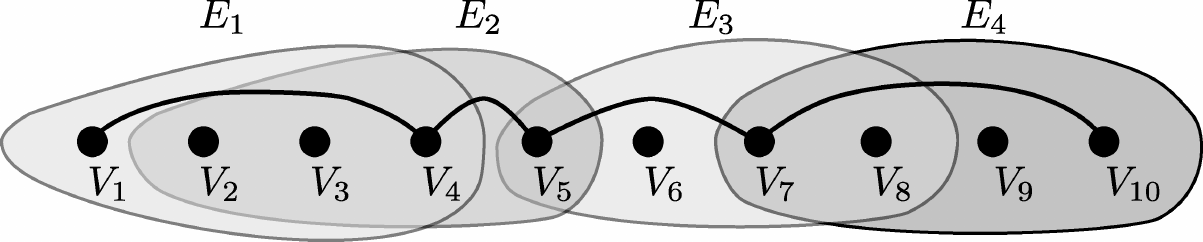}
\caption{A Berge-$P_4$ path $P=(E_1,E_2,E_3,E_4)$ between $V_1$ and $V_{10}$ with $P_4 = (  \{V_1,V_4\}, \{V_4,V_5\}, \{V_5,V_7\}, \{V_7,V_{10}\} )$. Observe that $P$ contains no shorter Berge path between $V_1$ and $V_{10}$.}
\label{fig:berge}
\end{figure}

\begin{lemma}\label{loose path between diamonds}
	Let $P=(E_1,\dots, E_\ell)$ be a Berge path in the cluster graph $R$ of $H$ with  $|V(R)|=m$ and edges of density $>\epsilon$.  Assume that $P$ contains no shorter Berge path that connects the two endpoints of $P$ and let $V_1,\dots,V_s$ be defined as above. Let $U_1\subset V_1,\dots, U_s\subset V_s$ be such that $|U_1|,|U_s|\geq \sqrt[k]{\epsilon}m$ and $|U_2|,\dots,|U_{s-1}|\geq 2\sqrt[k]{\epsilon}m$. 

Then, there is a loose path in $H$ that goes from a vertex in $V_1$ to a vertex in $V_s$ using only vertices from $\bigcup_{i=1}^s U_i$.
\end{lemma}

\begin{proof}
	Let us use induction on $\ell$. For $\ell=1$ then the statement follows from $\ep$-regularity. Now assume the statement holds for $\ell-1$ and consider such a path $P$ of length $\ell$. Note that by construction of the sequence of vertices $V_1, \ldots, V_s$, we have $V_{s-k+1}\in E_{\ell-1}\cap E_\ell$. Further, $P-E_\ell$ is a shortest Berge path connecting $V_1$ to $V_{s-k+1}$ and $U_1,\dots, U_{s-k+1}$ satisfy the requirements of the theorem, so by the inductive hypothesis, there is a loose path from $V_1$ to $V_{s-k+1}$ contained in $\bigcup_{i=1}^{s-k+1}U_i$. Let $v_1\in V_{s-k+1}$ be the last vertex in this path. If we replace $U_{s-k+1}$ with $U_{s-k+1}\setminus\{v_1\}$, then we can apply the inductive hypothesis again and find another loose path from $V_1$ to $V_{s-k+1}$ that does not use $v_1$. Since $|U_{s-k+1}|\geq 2\sqrt[k]{\epsilon}m$, we can repeat this process to find at least $\sqrt[k]{\epsilon}m$ loose paths from $V_1$ to $V_{s-k+1}$ which all have distinct endpoints in $V_{s-k+1}$. Let $X$ be the set of vertices in these at least $\sqrt[k]{\epsilon}m$ paths. Observe that $|X\cap U_s|=0$ since $P-E_\ell$ does not reach $U_s$, and for each $s-k+2\leq i\leq s-1$ we have that$|X\cap U_i|\leq \sqrt[k]{\epsilon}$ since each path hits at most one vertex from each cluster, so for all $s-k+2\leq i\leq s$, $|U_i\setminus X|\geq \sqrt[k]{\epsilon}m$ and $|U_{s-k+1}\cap X|\geq \sqrt[k]{\epsilon}m$. So there is an edge in $H[U_{s-k+1}\cap X,U_{s-k+2}\setminus X,\ldots,U_s\setminus X]$. This edge, along with whichever path from $V_1$ to $V_{s-k+1}$ is incident to it gives us the desired loose path from $V_1$ to $V_s$.
\end{proof}

Now we prove the main result of this section.  

\begin{proof}[Proof of Theorem \ref{thm:cycles1}]

Let $r$, $k$, $\alpha$, and $\psi = \psi(r,k)$ be given. Set $\alpha^* = \alpha/2$.  Let $\epsilon^*$ and $t_0$ be the constants guaranteed by the values of $k$, $r$, and $\alpha^*$. Let $\ep < \min\{ 1/(2r), (\alpha/(14 \psi))^k, \ep^*\}$. 
Thus, if $\Gamma$ is a $k$-uniform hypergraph on $t\geq t_0$ vertices with $e(\Gamma)\geq (1-\ep)\binom{t}{k}> (1-\ep^*)\binom{t}{k}$, then $\Gamma$ contains a monochromatic connected loose cycle packing on at least $(\psi-\alpha/2)t$ vertices for every $r$-edge-coloring.

Apply Theorem \ref{regularity} with $k$, $r$, $t_0$, $\ep$, as above, and $D=2$.  Let $\eta$, $T_0$ and $N_0$ be the constants that arise. Let $n\geq N_0$ be large. Let $H=\mathcal{H}^{(k)}(n,p)$. Color the edges of $H$ with $r$ colors. Let $H_i$ be the sub-hypergraph of $H$ colored $i$ for $1\leq i\leq r$. 
As in the proof of Theorem~\ref{thm:main1} it follows from the Chernoff bound~\eqref{eq:chernoff} that each $H_i$ is $(\eta,p,D)$-upper uniform. Then Theorem \ref{regularity} gives us that $H_1,\dots,H_r$ admits an $(\epsilon,p)$-regular partition.
	
	Let $R$ be the $k$-uniform cluster graph of $H$ with clusters $V_1,\dots,V_t$ each of size $m=n/t$. Color each edge $\{V_{i_1},\dots,V_{i_k}\}$ of $R$ with a majority color in the $k$-partite graph $H[V_{i_1},\dots,V_{i_k}]$. Let $H'$ be the sub-hypergraph colored by this majority color in $H[V_{i_1},\dots,V_{i_k}]$. Then again as in the proof of Theorem~\ref{thm:main1} we have $d_{H',p}(V_{i_1},\dots,V_{i_k}) \ge \epsilon$.
	
	Since $t\geq t_0$, and $|E(R)|\geq (1-\epsilon)\binom{t}{k}$, $R$ contains a monochromatic connected loose cycle packing, $C_1,\dots,C_{\ell}$, on at least $(\psi-\alpha/2)t$ clusters. Let $f = \sum_{i=1}^\ell |E(C_i)| \geq (\psi-\alpha/2)t/(k-1)$ be the number of edges in the packing. Let $E_1,\dots, E_f$ be an enumeration of these edges. There is a bijection between the edges in the cycle packing and the vertices of degree $2$ in the packing. We may assume without loss of generality that $V_i$ is the vertex corresponding to $E_i$ for each $1\leq i\leq f$. Partition $V_i$ into two equal sized sets $V_{i,1}$ and $V_{i,2}$ for each $1\leq i\leq f$.
	
	Fix $1\le i\leq f$ for a moment. Let $V_j\neq V_i$ be the second vertex of degree $2$ in $E_i$. Set $X_1=V_{i,1}$, $X_k=V_{j,2}$ and let $X_2,\dots,X_{k-1}$ be the clusters in $E_i\setminus\{V_i,V_j\}$. Set $\zeta=\sqrt[k]{\epsilon}$. Then Lemma~\ref{path in diamond} gives us that there is a loose path, call it $P_i$, on exactly $(1-4\sqrt[k]{\epsilon})m-2$ edges in $H'[X_1,\dots,X_k]$. We can do this for each $1\le i\leq f$ to find $f$ loose paths that we call \emph{long}. Let $U_{i,1}$ and $U_{i,2}$ be the first and last $\sqrt[k]{\epsilon}m$ vertices along $P_i$ in $V_i\cap V(P_i)$ respectively. 

	For each $V_i$, let $\tau(V_i)$ denote the set of vertices in $V_i$ that are not in any of the $f$ long paths. Notice that if $V_i$ is in the cycle packing, $\tau(V_i)\ge 4\sqrt[k]{\epsilon}m $ by the length of each $P_i$ and if $V_i$ is not in the cycle packing, $\tau(V_i)=|V_i| = m \ge 4\sqrt[k]{\epsilon}m$. Let $\mathcal{L}$ be the monochromatic component of $R$ containing the cycle packing. 
	
We now show how to find a loose path that connect the end of the path $P_i$ to the beginning of $P_{i+1}$ for each $1\leq i\leq f$.
Assume we have found vertex disjoint loose paths $Q_1,\dots,Q_{i-1}$ in $H'$ such that $Q_j$ connects a vertex in $U_{j,2} \subseteq V_{j}$ with a vertex in $U_{j+1,1}\subseteq V_{j+1}$ and does not intersect any other vertices in $\bigcup_{i=1}^{f}V(P_i)$. Furthermore, let $\mathcal{V}_{i-1} = \bigcup_{j=1}^{i-1} V(Q_j)$ and assume that $|\mathcal{V}_{i-1}\cap \tau(V_j)|\leq 2(i-1)$ for each $1\le j\le t$. Now we will construct a path $Q_i$ as follows.

Let $\Pi$ be the shortest Berge path from $V_i$ to $V_{i+1}$ in $\mathcal{L}$. Let $s$ be the number of clusters in $\Pi$. Let $V_{j_1}=V_i$ and $V_{j_s}=V_{i+1}$ and let $V_{j_2},\dots,V_{j_{s-1}}$  be the other $s-2$ clusters in $\Pi$. Let $U_{j_1}=U_{i,2}$, $U_{j_s}=U_{i+1,1}$ and $U_{j_2}=\tau(V_{j_2})\setminus\mathcal{V}_{i-1},\dots,U_{j_{s-1}}=\tau(V_{j_{s-1}})\setminus\mathcal{V}_{i-1}$. Then we have $|U_{j_1}|,|U_{j_s}|\geq \sqrt[k]{\epsilon}m$ (by definition of $U_{i,2}$ and $U_{i+1,1}$) and for each $1\le \ell \le s-1$,
\[
|U_{j_\ell}| \ge \tau(V_{j_{\ell}}) - 2f \ge 4\sqrt[k]{\epsilon}m - 2\sqrt[k]{\epsilon}m = 2\sqrt[k]{\epsilon}m,
\]
since $f \le \sqrt[k]{\epsilon}m$ for sufficiently large $n$ (and so $m$).
Hence, by Lemma \ref{loose path between diamonds} we have a loose path $Q_i$ from $U_{i,2}$ to $U_{i+1,1}$ that is disjoint from each short loose path found previously and that does not intersect any of the long paths except $P_i$ and $P_{i+1}$ at a single vertex in $U_{i,2}$ and $U_{i+1,1}$, respectively. Finally, observe that for $\mathcal{V}_{i} = \bigcup_{j=1}^{i} V(Q_j)$ we still have that 
\[
|\mathcal{V}_{i}\cap \tau(V_j)| = |\mathcal{V}_{i-1}\cap \tau(V_j)| + |V(Q_i)\cap \tau(V_j)| \le 2(i-1)+2=2i
\]
for each $1\le j\le t$, since if $Q_i$
contained three or more vertices from a single cluster $V_j$, the corresponding Berge path $\Pi$ would not have been a shortest Berge path.

Notice then these short paths $Q_i$'s together with the long paths $P_i$'s give us a long loose cycle, $C$. By our choice for each $U_{i,1}$ and $U_{i,2}$, our cycle uses all the edges from each $P_i$ except at most $\sqrt[k]{\epsilon}m$ from each the beginning of the path and the end. There are $f$ long paths, so the total number of edges in our cycle is at least
\[
|E(C)|\geq\sum_{i=1}^f(|E(P_i)|-2\sqrt[k]{\epsilon}m)\geq f(1-7\sqrt[k]{\epsilon})m\geq \frac{(\psi-\alpha/2)t}{k-1}(1-7\sqrt[k]{\epsilon})m\geq\frac{\psi-\alpha}{k-1}n,
\]
where the last inequality uses $\ep < (\alpha/(14 \psi))^k$.
This immediately implies that $|V(C)|\geq (\psi-\alpha)n$, as required.
\end{proof}

The following result was proved by Gy\'arf\'as, S\'ark\"ozy, and Szemer\'edi \cite{GSS2008diamond}.  

\begin{theorem}[\cite{GSS2008diamond}]\label{diamonds}
	Suppose that $k$ is fixed and the edges of an almost complete $k$-uniform hypergraph on $n$ vertices are $2$-colored. Then there is a monochromatic connected diamond matching $cD_k$ such that $|V(cD_k)| =  (1+o(1))\frac{2k-2}{2k-1}n$.
\end{theorem}
\noindent
Combining Theorem \ref{diamonds} with Theorem \ref{thm:cycles1}, we immediately get a proof of Corollary \ref{cor:cycles}, which is best possible as observed in \cite{GSS2008diamond}.

\begin{observation}
There exists a $2$-coloring of the edges of the complete $k$-uniform hypergraph $K^{k}_n=(V,E)$ such that the longest loose cycle covers no more than $\frac{2k-2}{2k-1}n$ vertices.
\end{observation}

\begin{proof}
Choose any set $S$ of $\frac{2k-2}{2k-1}n$ vertices, and color all the edges completely inside $S$ red, and the rest of the edges blue. Clearly the longest red loose cycle has no more than $\frac{2k-2}{2k-1}n$ vertices. For the longest blue path, notice that each vertex in $V\setminus S$ can be in at most two edges, and each edge must contain one such vertex. Thus the longest possible cycle is of length $\frac{2}{2k-1}n$ or order $\frac{2k-2}{2k-1}n$.
\end{proof}

\section{$k$-uniform hypergraphs colored with $k$ colors}\label{sec:4colored}

\begin{theorem}\label{kunikcolors}
Let $k\geq 3$ and let $0<\ep< 16^{-k}$. Then there exists some $n_0$ such that for any $k$-uniform hypergraph $G$ on $n>n_0$ vertices with $e(G)>(1-\ep)\binom{n}{k}$, we have $\mc_k(G)\geq (1-8\sqrt[k]{\ep})n$.
\end{theorem}

\begin{proof}
Let $G$ be a $k$-uniform hypergraph on $n$ vertices with $e(G)> (1-\ep)\binom{n}{k}$ and suppose that $\mc_k(G)< (1-8\sqrt[k]{\ep})n$. Let $G$ be $k$-colored in such a way that the largest monochromatic component of $G$ is of size $\mc_k(G)$.

First, for each $i\in [k]$, we  find a partition $\{A_i, B_i\}$ of $V(G)$ such that $(1-4\sqrt[k]{\ep})n\geq |A_i|, |B_i|\geq 4\sqrt[k]\ep n$ and there are no edges of color $i$ which are incident with both a vertex in $A_i$ and a vertex in $B_i$. To find such a partition, let $C_i$ be the largest component of color~$i$. By assumption $|V(C_i)|<(1-8\sqrt[k]{\epsilon})n<(1-4\sqrt[k]{\epsilon})n$. If $|V(C_i)|\geq 4\sqrt[k]{\epsilon}n$, then we can let $A_i=V(C_i)$ and $B_i=V(G)\setminus V(C_i)$. If $|V(C_i)|\leq 4\sqrt[k]{\epsilon}n$, then some union of components of color $i$, call it $C^*_i$, will have the property that $4\sqrt[k]\epsilon n\leq |V(C^*_i)| \le 4\sqrt[k]\epsilon n+4\sqrt[k]\epsilon n \leq n/2$,
since $\ep< 1/16^k$. (We can always choose $C_i^*$ such that $|V(C^*_i)| \ge 4\sqrt[k]\epsilon n$, since we treat vertices that are isolated with respect to color $i$ as separate components.) Here we can let $A_i=V(C^*_i)$ and $B_i=V(G)\setminus V(C^*_i)$.

We can also assume that $|A_1\cap A_2|\geq 2\sqrt[k]\ep n$ and $|B_1\cap B_2|\geq 2\sqrt[k]\ep n$. Indeed, if $|A_1 \cap A_2|<2\sqrt[k]\ep n$, then $|B_1 \cap A_2|>2\sqrt[k]\ep n$.  In this case, we can switch the roles of $A_2$ and $B_2$ and now we have $|B_1\cap B_2|>2\sqrt[k]\ep n$ and  $|A_1\cap A_2| = |A_1| - |A_1\cap B_2|>|A_1|-2\sqrt[k]\ep n\geq 2\sqrt[k]\ep n$. Similarly, if $|B_1\cap B_2|<2\sqrt[k]\ep n$, switching the roles of $A_2$ and $B_2$ gives us the desired inequalities. Thus we can always  assume that $|A_1\cap A_2|\geq 2\sqrt[k]\ep n$ and $|B_1\cap B_2|\geq 2\sqrt[k]\ep n$. 

 Furthermore, we can assume that $|A_1\cap A_2\cap \dots \cap A_k|\geq 2\sqrt[k]\ep n/2^{k-2}$. Indeed, we already have that $|A_1\cap A_2|\geq 2\sqrt[k]\ep n$. Now for a fixed $i$ satisfying $3\leq i\leq k$ assume that $|A_1\cap\dots\cap A_{i-1}|\geq 2\sqrt[k]{\ep}n/2^{i-3}$. Since $\{A_i\cap(A_1\cap \dots A_{i-1}), B_i\cap (A_1\cap \dots A_{i-1})\}$ is a partition of $A_1\cap \dots \cap A_{i-1}$, if $A_i$ does not cover at least half of vertices in $A_1\cap \dots \cap A_{i-1}$, then $B_i$ does, so switching the roles of $A_i$ and $B_i$ if necessary will give us that $|A_1\cap\dots\cap A_{i}|\geq 2\sqrt[k]{\ep}n/2^{i-2}$. Thus, we can recursively choose $A_i$ such that $|A_1\cap A_2\cap \dots \cap A_k|\geq 2\sqrt[k]\ep n/2^{k-2}$.

Now observe that the number of $k$-tuples of distinct vertices $(x_1, \dots, x_k)$ where $x_1\in A_1\cap A_2\cap \dots \cap A_k$, $x_2\in B_1\cap B_2$, and $x_i\in B_i$ for all $3\leq i\leq k$, is at least 
\begin{equation}\label{eq:ktuple}
2\sqrt[k]\ep n/2^{k-2}\cdot (2\sqrt[k]\ep n-1)\cdot (4\sqrt[k]\ep n-2)\cdots (4\sqrt[k]\ep n-k+1) =2^k \ep n^k +O(n^{k-1}).
\end{equation}
Note that any such $k$-tuple intersects every set $A_1, \dots, A_k, B_1, \dots, B_k$ and thus no such $k$-tuple can be an edge of $G$ since if it were an edge of $G$, it would be colored, say by color $i$, and thus cannot contain vertices from both $A_i$ and $B_i$. Since there are less than $\ep\binom{n}{k}$ non-edges in $G$, there are less than $\ep n^k$ ordered $k$-tuples of distinct vertices $(x_1, \dots, x_k)$ where $\{x_1, \dots, x_k\}$ is not an edge, which contradicts \eqref{eq:ktuple} for $n$ sufficiently large.
\end{proof}

\section{$k$-uniform hypergraphs colored with $k+1$ colors}\label{sec:5colored}

The following proof makes use of some technical lemmas which can all be found in Section~\ref{sec:lemmas}.

\begin{theorem}\label{kunik+1colors}
Let $k\geq 3$ and $0<\ep <\frac{1}{k^{5k}2^{8k}}$. If $G$ is a $k$-uniform hypergraph on $n$ vertices with $\delta(G)\geq (1-\ep)\binom{n-1}{k-1}$, then $\mc_{k+1}(G)\geq (\frac{k}{k+1}-\sqrt{k^k\ep})n$.
\end{theorem}

\begin{proof}
Suppose that the edges of $G$ are $(k+1)$-colored.  First note for all $v\in V(G)$, the link graph of $v$ is a $(k+1)$-colored $(k-1)$-uniform hypergraph on $n-1$ vertices with at least $(1-\ep)\binom{n-1}{k-1}$ edges and thus by Lemma \ref{lem:1core} (applied with $k-1$ in place of $k$ and $\ell=2$), we get a monochromatic $1$-core in the link graph of $v$ of order at least $(\frac{k-1}{k+1}-\sqrt{\ep})(n-1)$. Since all these vertices are connected via $v$, this gives us a monochromatic component containing $v$ of order at least $(\frac{k-1}{k+1}-\sqrt{\ep})(n-1)+1\geq (\frac{k-1}{k+1}-\sqrt{\ep})n.$ Now let $C_1$ be the largest monochromatic component in $G$, and note that by the calculation above, 
\[
|C_1|\geq \rbrac{\frac{k-1}{k+1}-\sqrt{\ep}}n.
\]

For ease of reading, set $\alpha=k^k\ep$. For any vertex $v\in V(G)\setminus C_1$, we have
\begin{equation}\label{eq:degtoC}
d(v, C_1)\geq \binom{|C_1|}{k-1}-\ep\binom{n-1}{k-1}\geq \rbrac{1-\frac{k\ep}{(\frac{k-1}{k+1}-\sqrt{\ep})^{k-1}}}\binom{|C_1|}{k-1}\geq (1-\alpha)\binom{|C_1|}{k-1},
\end{equation}
where the second inequality holds by Observation \ref{obs:binomcalc} (applied with $U=C_1$ and $\lambda=\frac{k-1}{k+1}-\sqrt{\ep}$) and the last inequality holds since $\frac{k-1}{k+1} - \sqrt{\ep} \geq  \frac 1k$. 

Now choose a monochromatic component $C_2$ so that $|C_1\cap C_2|$ is maximized. We claim that 
\[
|C_1\cap C_2|\geq \rbrac{\frac{k-1}{k}-\sqrt{\alpha}}|C_1|.
\]
To see this, let $v\in V(G)\setminus C_1$ and note that the link graph of $v$ restricted to $C_1$ is a $(k-1)$-uniform hypergraph on $|C_1|$ vertices with at least $(1-\alpha)\binom{|C_1|}{k-1}$ edges (by \eqref{eq:degtoC}).
Furthermore, note that the the link graph of $v$ restricted to $C_1$ only uses $k$ colors since the color of $C_1$ cannot show up by our choice of $v\not\in C_1$.
Thus by Lemma \ref{lem:1core} (with $k$ replaced by $k-1$ and $\ell=1$) there is a monochromatic component $C_2$ containing $v$ such that $|C_1\cap C_2|\geq \rbrac{\frac{k-1}{k}-\sqrt{\alpha}}|C_1|$, thus proving the claim.  

Now if $V\setminus C_1\subseteq C_2$, then we have 
$$|C_1|\geq |C_2|=n-|C_1|+|C_1\cap C_2|\geq n-|C_1|+\rbrac{\frac{k-1}{k}-\sqrt{\alpha}}|C_1|,$$
which implies
$$|C_1|\geq \frac{n}{\frac{k+1}{k}+\sqrt{\alpha}}\geq \rbrac{\frac{k}{k+1}-\sqrt{\alpha}}n;$$
and we are done, so suppose not.  

We now show that $|C_1\setminus C_2|$ must be small (in other words, $|C_1\cap C_2|$ must be big); more precisely, we will show that 
\begin{equation}\label{C1-C2}
|C_1\setminus C_2|\leq 8\sqrt[2(k-1)]{\alpha}|C_1|.  
\end{equation}
Let $v\in V(G) \setminus (C_1\cup C_2)$ and consider the link graph $L'$ obtained from $L(v, C_1)$ by deleting all edges entirely contained in $C_1 \setminus C_2$.

Notice that $L(v,C_1)$ is a $(k-1)$-uniform hypergraph with at least $(1-\alpha)\binom{|C_1|}{k-1}$ edges (by \eqref{eq:degtoC}). Furthermore, $L'$ is $(k-1)$-colored since it cannot contain any edges whose color matches $C_1$, and if there was an edge in $L(v,C_1)$ whose color matched $C_2$, it would need to be contained in $C_1\setminus C_2$, but such edges do not exist in $L'$. Thus by Lemma \ref{lem:1core_AB} (applied with $G'=L(v,C_1)$, $G=L'$, $A=C_1\cap C_2$ and $B=C_1\setminus C_2$) either there is a monochromatic component containing $v$ and at least 
$(1-8\sqrt[2(k-1)]{\alpha})|C_1|$ vertices of $C_1$, in which case \eqref{C1-C2} holds since $C_2$ was assumed to be the component that maximizes $|C_1\cap C_2|$, or there is a monochromatic component containing $v$ and at least 
\begin{equation}\label{C1capC2}
|C_1\cap C_2|-\sqrt{\alpha}|C_1|+\rbrac{\frac{k-2}{k-1}-\frac{\sqrt{k-1}\sqrt[4(k-1)]{\alpha}}{2^{k-2}}}|C_1\setminus C_2|
\end{equation}
vertices of $C_1$.  Observe that it cannot be the case that
\[
\rbrac{\frac{k-2}{k-1}-\frac{\sqrt{k-1}\sqrt[4(k-1)]{\alpha}}{2^{k-2}}}|C_1\setminus C_2|>\sqrt{\alpha}|C_1|,
\]
since this and \eqref{C1capC2} would imply that we have a monochromatic component containing $v$ and more than $|C_1\cap C_2|$ vertices of $C_1$, which contradicts the choice of $C_2$. So it must be that 
$\rbrac{\frac{k-2}{k-1}-\frac{\sqrt{k-1}\sqrt[4(k-1)]{\alpha}}{2^{k-2}}}|C_1\setminus C_2|\leq \sqrt{\alpha}|C_1|$, which implies
\begin{equation}\label{C_1setminusC_2}
|C_1\setminus C_2|\leq \frac{\sqrt{\alpha}|C_1|}{\frac{k-2}{k-1}-\frac{\sqrt{k-1}\sqrt[4(k-1)]{\alpha}}{2^{k-2}}}.
\end{equation}
We further claim that
\begin{equation}\label{C_1setminusC_2 bound}
\frac{\sqrt{\alpha}|C_1|}{\frac{k-2}{k-1}-\frac{\sqrt{k-1}\sqrt[4(k-1)]{\alpha}}{2^{k-2}}}<  8\sqrt[2(k-1)]{\alpha}|C_1|.
\end{equation}
Indeed, first note that since $\alpha<2^{-8k}$ and $k\ge 3$, we get $\sqrt[4(k-1)]{\alpha} < 2^{-2k/(k-1)} \le 1/8$, 
$\frac{\sqrt{k-1}}{2^{k-2}} \le 1$ and $\frac{k-2}{k-1} \ge 1/2$. Consequently,
\[
\frac{k-2}{k-1}-\frac{\sqrt{k-1}\sqrt[4(k-1)]{\alpha}}{2^{k-2}}
\ge \frac{1}{2} - \frac{1}{8} > \frac{1}{8},
\]
which together with $\sqrt{\alpha} \le \sqrt[2(k-1)]{\alpha}$ yields~\eqref{C_1setminusC_2 bound}. Thus by \eqref{C_1setminusC_2} and \eqref{C_1setminusC_2 bound},
\[
|C_1\setminus C_2|<8\sqrt[2(k-1)]{\alpha}|C_1|
\]
as desired.
 
Now that we have established \eqref{C1-C2}, we will show that any vertex in $V(G)\setminus C_1$ is in some monochromatic component that has large intersection with $C_1$. More precisely, for all $v\in V(G)\setminus C_1$, there exist some monochromatic component $C\ni v$ such that $|C\cap C_1|\geq (1-9\sqrt[2(k-1)]{\alpha})|C_1|$. Indeed, let $v\in V(G)\setminus C_1$. If $v\in C_2$, then \eqref{C1-C2} implies that $C=C_2$ suffices. Otherwise, 
as in \eqref{C1-C2} and \eqref{C1capC2} (with $C_2$ replaced by $C$) we get that there is a monochromatic component containing $v$ and at least 
\begin{align*}
|C_1\cap C| - \sqrt{\alpha}|C_1| &= |C_1| - |C_1\setminus C| - \sqrt{\alpha}|C_1|\\
&\ge (1-8\sqrt[2(k-1)]{\alpha}-\sqrt{\alpha})|C_1|\geq (1-9\sqrt[2(k-1)]{\alpha})|C_1|
\end{align*}
vertices of $C_1$. Also observe that since $(1-9\sqrt[2(k-1)]{\alpha})|C_1| > |C_1|/2$, any two such components~$C$ intersect with each
other.

Since there are only $k$ possible colors for edges between $C_1$ and $V(G)\setminus C_1$, there must exist at least $(n-|C_1|)/k$ vertices in $V(G)\setminus C_1$ that are all in some monochromatic component~$C^*$ together along with at least $(1-9\sqrt[2(k-1)]{\alpha})|C_1|$ vertices of $C_1$. Thus,

\[
|C_1|\geq |C^*|\geq (n-|C_1|)/k+(1-9\sqrt[2(k-1)]{\alpha})|C_1|,
\]
where the first inequality holds by the maximality of $C_1$.  So solving for $C_1$ we finally obtain that
\[
|C_1|\geq \frac{n}{k\left(\frac{1}{k}+9\sqrt[2(k-1)]{\alpha}\right)}>\left(\frac{k}{k+1}-\sqrt{\alpha}\right)n,
\]
where the final inequality follows from the upper bound on $\ep$ (and so on $\alpha$).
\end{proof}

\begin{corollary}\label{cor:degtoedges}
Let $k\geq 3$ and $0<\ep <\frac{1}{k^{10k+2}2^{16k+2}}$. If $G$ is a $k$-uniform hypergraph on $n$ vertices with $e(G)\geq (1-\ep)\binom{n}{k}$, then $\mc_{k+1}(G)\geq \rbrac{\frac{k}{k+1}-2k^{\frac{k+1}2}\sqrt[4]{\ep}}n$.
\end{corollary}

\begin{proof}
Apply Observation \ref{high degree} to get a subgraph $G'=(V', E')$ with $|V'|\geq (1-\sqrt{\ep})n$ and $\delta(G')\geq (1-2k\sqrt{\ep})\binom{|V'|-1}{k-1}$.  Now by Theorem \ref{kunik+1colors},  we have
\[\mc_{k+1}(G)\geq\mc_{k+1}(G')\geq \rbrac{\frac{k}{k+1}-\sqrt{2}k^{\frac{k+1}2}\sqrt[4]{\ep}}(1-\sqrt{\ep})n\geq \rbrac{\frac{k}{k+1}-2k^{\frac{k+1}2}\sqrt[4]{\ep}}n.\]
\end{proof}

\section{Technical lemmas}\label{sec:lemmas}

The first two observations are simple counting arguments, but it is convenient for us to state them explicitly.

\begin{observation}\label{obs:binomcalc}
Let $n\geq k\geq 2$, let $0<\ep,  \lambda\leq 1$. If $|U|=\lambda n\geq k^2$, then $$\binom{|U|-1}{k-1}-\ep\binom{n-1}{k-1}\geq \rbrac{1-\frac{k\ep}{\lambda^{k-1}}}\binom{|U|-1}{k-1}.$$
\end{observation}

\begin{proof}
First we show that $\lambda n\geq k^2$ implies that $\frac{n-(k-1)}{\lambda n-(k-1)}\leq \frac{\sqrt[k-1]{k}}{\lambda}$. By Bernoulli's inequality, $(1+x)^r \ge 1+xr$ for all real numbers $r\ge 1$ and $x\ge-1$ (see, e.g., inequality 58 in \cite{HLP1952}) we get that
\[
k^{k/(k-1)} = \left(1+(k-1)\right)^{k/(k-1)} \ge 1+k,
\]
and equivalently, $\sqrt[k-1]{k} \ge 1+1/k$. Thus,
\begin{align*}
\sqrt[k-1]{k} \ge 1 + \frac{1}{k} = 1 + \frac{k-1}{k(k-1)} &\ge 1 + \frac{k-1}{k^2 - k+1} \\
&\ge 1 + \frac{k-1}{\lambda n - k+1} = \frac{\lambda n}{\lambda n -(k -1)} \ge \frac{\lambda n - \lambda(k-1)}{\lambda n -(k -1)}.
\end{align*}
Now,
\begin{align*}
\binom{|U|-1}{k-1}-\ep\binom{n-1}{k-1}
&= \binom{|U|-1}{k-1}-\ep\binom{|U|-1}{k-1}\cdot \frac{\binom{n-1}{k-1}}{\binom{|U|-1}{k-1}}\\
&\geq \binom{|U|-1}{k-1}-\ep\binom{|U|-1}{k-1}\left(\frac{n-(k-1)}{\lambda n-(k-1)}\right)^{k-1} \\
&\geq  \rbrac{1-\frac{k\ep}{\lambda^{k-1}}}\binom{|U|-1}{k-1}.
\end{align*}
\end{proof}

\begin{observation}\label{high degree}
Let $k\geq 2$, let $0<\eta<1$, and let $0<\ep\leq (1-\sqrt[k-1]{1/2})^{1/(1-\eta)}$.  Let $G=(V,E)$ be a $k$-uniform hypergraph on $n$ vertices with $e(G)\geq (1-\ep)\binom{n}{k}$.  For all $U\subseteq V$ all but at most  $\ep^{1-\eta} n$ vertices have
\begin{equation}\label{obs:2:eq:1}
d(v, U)\geq \binom{|U|-1}{k-1}-\ep^{\eta}\binom{n-1}{k-1}.
\end{equation}
In particular, if $(1-\ep^{1-\eta})n \ge k^2$, then there exists an induced subgraph $G'=(V', E')$ with $|V'|\geq (1-\ep^{1-\eta})n$ and $\delta(G')\geq (1-2k\ep^\eta)\binom{|V'|-1|}{k-1}$.
\end{observation}

\begin{proof}
Let $V^*=\{v\in V: d(v)< \binom{n-1}{k-1}-\ep^{\eta}\binom{n-1}{k-1}\}$.  We have 
\[
|V^*|\frac{\ep^{\eta}}{k}\binom{n-1}{k-1}\leq e(\bar{G})\leq \ep\binom{n}{k},
\]
 which implies $|V^*|\leq \ep^{1-\eta}n$.  So let $U\subseteq V$ with $|U|= \lambda n\geq k^2$ and note that for all $v\in V\setminus V^*$ we have 
\begin{equation}\label{eq:dvU}
d(v, U) \geq \binom{|U|-1}{k-1}-\ep^{\eta}\binom{n-1}{k-1},
\end{equation}
as required.


To see the final statement, let $V'=V\setminus V^*$ and note that $|V'|\geq (1-\ep^{1-\eta})n$ by the calculation of $|V^*|$ above.  
From \eqref{eq:dvU} (applied with $U=V'$) we have
\[
\delta(G[V']) \ge \binom{|V'|-1}{k-1}-\ep^{\eta}\binom{n-1}{k-1}.
\]
Now Observation~\ref{obs:binomcalc} with $\lambda  = 1-\ep^{1-\eta}$ yields
\[
\delta(G[V']) \geq \rbrac{1-\frac{k\ep^\eta}{(1-\ep^{1-\eta})^{k-1}}}\binom{|V'|-1}{k-1}\geq (1-2k\ep^\eta)\binom{|V'|-1}{k-1},
\]
where the last inequality holds since $1-\ep^{1-\eta}\geq \sqrt[k-1]{1/2}$.

\end{proof}

\begin{lemma}\label{lem:1core}
Let $k\geq 2$, $\ell\geq 1$, and $0<\ep<256^{-k}$.  Then there exists $n_0>0$, a sufficiently large integer, such that if $G=(V,E)$ is a $(k+\ell)$-colored $k$-uniform hypergraph on $n\ge n_0$ vertices with $e(G)>(1-\epsilon) \binom{n}{k}$, then $G$ contains a monochromatic 1-core on at least $(\frac{k}{k+\ell}-\sqrt{\ep})n$  vertices.
\end{lemma}

\begin{proof}
By~\eqref{obs:2:eq:1} in Observation \ref{high degree} (applied with  $\eta=1/2$, $\lambda = 1$, and $U=V$) at least $(1-\sqrt{\ep})n$ vertices have degree at least $(1-\sqrt{\ep})\binom{n-1}{k-1}$. Call these vertices~$V'$.

Suppose that the 1-core of each of the $k+\ell$ colors has order less than $(\frac{k}{k+\ell}-\sqrt{\ep})n$; that is, for each color $i\in [k+\ell]$, more than $\rbrac{\frac{\ell}{k+\ell}+\sqrt{\ep}}n$ vertices are not incident with an edge of color $i$. Let $X$ be the set of vertices that are not incident with at least $\ell+1$ colors. We claim that  $|X|\geq \frac{k+\ell}{k}\sqrt{\ep}n>\sqrt{\ep}n$.
Indeed, let $T\subseteq V\times [k+\ell]$ be the set of ordered pairs where $(v,t)\in T$ if the vertex $v$ is not incident with any edge of color $t$. Since for each $i\in[k+\ell]$, there are more than $\rbrac{\frac{\ell}{k+\ell}+\sqrt{\ep}}n$ vertices that are not incident with an edge of color $i$, we have that 
\begin{equation}\label{pigeonhole equation 1}
|T|\geq (k+\ell)\rbrac{\frac{\ell}{k+\ell}+\sqrt{\ep}}n=\ell n+\sqrt{\epsilon}(k+\ell)n.
\end{equation}
Now each vertex of $V\setminus X$ contributes no more than $\ell$ ordered pairs to the set $T$, and each vertex in $X$ can contribute no more than $k+\ell$, so we have that
\begin{equation}\label{pigeonhole equation 2}
|T|\leq (n-|X|)\ell+|X|(k+\ell)=\ell n+ k|X|.
\end{equation}
Combining \eqref{pigeonhole equation 1} and~\eqref{pigeonhole equation 2} gives us that $\ell n+k|X|\geq \ell n+\sqrt{\epsilon}(k+\ell)n$, which implies that $|X|\geq \frac{k+\ell}{k}\sqrt{\ep}n\geq \sqrt{\epsilon}n$, which proves the claim.

By the bounds on $|X|$ and $|V'|$ there exists $v\in V'\cap X$, and the link graph of $v$ is a  $(k-1)$-uniform hypergraph on $n-1$ vertices with at least $(1-\sqrt{\ep})\binom{n-1}{k-1}$ edges which is colored with at most $k-1$ colors and thus by Theorem \ref{kunikcolors}, there exists a monochromatic component on at least 
\[
(1-8\sqrt[2k]{\ep})(n-1)\geq \rbrac{\frac{k}{k+\ell}-\sqrt{\ep}}n
\]
 vertices, where the above inequality holds by the choice of $\ep$.  So, the 1-core of some color must have order at least $(\frac{k}{k+\ell}-\sqrt{\ep})n$.
\end{proof}

\begin{lemma}\label{lem:1core_AB}
Let $k\geq 2$ and $0< \ep\le 512^{-k}$. Then there exists a sufficiently large integer $n_0$ such that the following holds. Let $G$ be a hypergraph obtained from a $k$-uniform hypergraph $G'$ on $n\geq n_0$ vertices with $e(G')\geq (1-\ep)\binom{n}{k}$ by letting $\{A,B\}$ be a partition (chosen arbitrarily) of $V(G')$ with $|A|>\sqrt{\ep}n$, and deleting all edges which lie entirely inside $B$.  Then for any $k$-edge-coloring of $G$ there exists a monochromatic $1$-core in $G$ on at least 
\[
\min\cbrac{(1-8\sqrt[2k]{\ep})n, |A|-\sqrt{\ep}n+\rbrac{\frac{k-1}{k}-\frac{\sqrt{k}\sqrt[4k]{\ep}}{2^{k-1}}}|B|}
\]
vertices.  
\end{lemma}

\begin{proof}
By~\eqref{obs:2:eq:1} in Observation \ref{high degree} (applied to $G'$ with parameters $\eta=1/2$, $\lambda = 1$, and $U=V$) we get that all but at most $\sqrt{\ep}n$ vertices in $G'$ have degree at least $(1-\sqrt{\ep})\binom{n-1}{k-1}$.  Let $A'$ be the vertices in $A$ which have degree at least $(1-\sqrt{\ep})\binom{n-1}{k-1}$ in $G$ and note that since the degrees of the  vertices in $A$ are the same in $G$ as in $G'$, we have that $|A'|\geq |A|-\sqrt{\ep}n>0$.  First suppose that there exists a color $i\in [k]$ and a vertex $v\in A'$ such that $v$ is not contained in the 1-core of color $i$.  Then the link graph of $v$ in $G$ is a $(k-1)$-colored $(k-1)$-graph on $n-1$ vertices with at least $(1-\sqrt{\ep})\binom{n-1}{k-1}$ edges.  
By Theorem \ref{kunikcolors}, we have a monochromatic component in the link graph of $v$ of order at least $(1-8\sqrt[2k]{\ep})(n-1)$ which together with $v$ gives a monochromatic component (and so a 1-core) in $G$ of order $(1-8\sqrt[2k]{\ep})(n-1) + 1 \geq (1-8\sqrt[2k]{\ep})n$.
Otherwise every vertex in $A'$ is contained in the 1-core of every color.  Now if $|A'|\geq (1-8\sqrt[2k]{\ep})n$, then we have a 1-core of the desired size, so suppose $|A'|< (1-8\sqrt[2k]{\ep})n$, which implies 
$$|B|= n-|A|\geq n-|A'|-\sqrt{\ep}n\geq  4\sqrt[2k]{\ep}n.$$
Let $v\in A'$ and consider the link graph of $v$ restricted to $B$, which is a $(k-1)$-uniform hypergraph on $|B|$ vertices colored with $k$ colors which, by Observation \ref{obs:binomcalc} (applied with $U=B$, $\lambda = |B|/n \ge 4\sqrt[2k]{\ep}$), has at least
\begin{align*}
\binom{|B|}{k-1} - \sqrt{\ep} \binom{n-1}{k-1} &\ge
\rbrac{1-\frac{k\sqrt{\ep}}{\lambda^{k-1}}}\binom{|B|}{k-1}\\
&\geq \rbrac{1-\frac{k\sqrt{\ep}}{(4\sqrt[2k]{\ep})^{k-1}}}\binom{|B|}{k-1}= \rbrac{1-\frac{k\sqrt[2k]{\ep}}{4^{k-1}}}\binom{|B|}{k-1}
\end{align*}
edges. Now by Lemma \ref{lem:1core} (applied to the link graph of $v$ with $k$ replaced by $k-1$ and $\ell=1$), there is a monochromatic 1-core of color $i$ which contains at least $(\frac{k-1}{k}-\frac{\sqrt{k}\sqrt[4k]{\ep}}{2^{k-1}})|B|$ vertices of $B$, and since every vertex in $A'$ is in the 1-core of color $i$, the total size of the 1-core of color $i$ is at least 
\[
|A'|+\left(\frac{k-1}{k}-\frac{\sqrt{k}\sqrt[4k]{\ep}}{2^{k-1}}\right)|B|\geq |A|-\sqrt{\ep}n+\left(\frac{k-1}{k}-\frac{\sqrt{k}\sqrt[4k]{\ep}}{2^{k-1}}\right)|B|.
\]
\end{proof}

\section{Conclusion}\label{sec:conc}

The most satisfactory results of this paper are for $k$-colored or $(k+1)$-colored $k$-uniform hypergraphs. In order to obtain them, we extended Theorem~\ref{gthm1} to almost complete hypergraphs. However, for complete hypergraphs more is known.

\begin{theorem}[F\"uredi and Gy\'arf\'as~\cite{FG1991}]\label{fgthm}
Let $k,r\geq 2$ and let $q$ be the smallest integer such that $r\leq q^{k-1}+q^{k-2}+\dots+q+1$.  Then $\mc_r(K^k_n)\geq \frac{n}{q}$. 
This is sharp when $q^k$ divides $n$, $r=q^{k-1}+q^{k-2}+\dots+q+1$, and an affine space of dimension $k$ and order $q$ exists.
\end{theorem}



\begin{theorem}[Gy\'arf\'as and Haxell~\cite{GH2009}]\label{ghthm}
$\mc_{5}(K^3_n)\geq 5n/7$ and $\mc_{6}(K^3_n)\geq 2n/3$.  Furthermore these are tight when $n$ is divisible by $7$ and $6$ respectively.  
\end{theorem}

\noindent
It would be interesting to extend these results to the nearly complete setting, from which the random analogs would then follow by applying Theorem \ref{thm:main}.  It is worth noting that these results follow by applying F\"uredi's fractional transversal method which does not seem to extend to non-complete graphs in a straightforward manner.

\medskip

\textbf{Acknowledgment} We are grateful to all referees for their detailed comments on an earlier version of this paper.


\providecommand{\bysame}{\leavevmode\hbox to3em{\hrulefill}\thinspace}
\providecommand{\MR}{\relax\ifhmode\unskip\space\fi MR }
\providecommand{\MRhref}[2]{%
  \href{http://www.ams.org/mathscinet-getitem?mr=#1}{#2}
}
\providecommand{\href}[2]{#2}

\end{document}